\documentclass[12pt,a4]{amsart}

\calclayout

\usepackage{xcolor}

\usepackage{amsmath, amssymb, amsthm, amscd,color,comment}

\usepackage[all,cmtip]{xy}

\usepackage{mathrsfs}

\usepackage{tabularx}
\usepackage{booktabs}
\usepackage{enumitem}
\usepackage{comment}

\usepackage{tikz}
\usetikzlibrary{calc}

\usepackage[labelfont=bf,format=plain,justification=raggedright,singlelinecheck=false]{caption}

\newcommand{\F}{\mathbb{F}}

\newcommand{\aut}{\mathrm{Aut}}

\newcommand{\Fqq}{\mathbb{F}_{q^2}}

\newcommand{\mZ}{\mathbb{Z}}

\newcommand{\cF}{\mathcal{F}}
\newcommand{\cH}{\mathcal{H}}

\numberwithin{equation}{section}

\theoremstyle{plain}
\newtheorem{theorem}[equation]{Theorem}
\newtheorem{corollary}[equation]{Corollary}
\newtheorem{lemma}[equation]{Lemma}
\newtheorem{proposition}[equation]{Proposition}

\theoremstyle{definition}

\theoremstyle{remark}

\usepackage{bookmark}

\usepackage{hyperref}

\begin{document}

\title{Non-isomorphic maximal function fields of genus $q-1$}

\thanks{$^1$ Technical University of Denmark, Kgs. Lyngby, Denmark, jtni@dtu.dk}



\author{Jonathan Niemann$^1$}

\begin{abstract}
    The classification of maximal function fields over a finite field is a difficult open problem, and even determining isomorphism classes among known function fields is challenging in general. We study a particular family of maximal function fields defined over a finite field with $q^2$ elements, where $q$ is the power of an odd prime. When $d := (q+1)/2$ is a prime, this family is known to contain a large number of non-isomorphic function fields of the same genus and with the same automorphism group. We compute the automorphism group and isomorphism classes also in the case where $d$ is not a prime.
\end{abstract}

\maketitle

\section{Introduction}

Function fields over finite fields with many rational places have been studied extensively in the past decades, partly due to the role they play in constructing error-correcting codes with good parameters. The number of rational places of such a function field is bounded from above by the Hasse-Weil bound. In fact, if $\cF$ is a function field defined over $\Fqq$, then 
$$
    N(\cF) \leq q^2 + 1 + 2g(\cF)q,
$$
where $g(\cF)$ is the genus of $\cF$ and $N(\cF)$ is the number of places of degree one over $\Fqq$. A function field attaining this bound is called $\Fqq$-maximal, and the classification of all $\Fqq$-maximal function fields is a big open problem. 

Perhaps the most studied example of a $\Fqq$-maximal function field is the Hermitian function field
$$
    \cH := \Fqq(x,y) \ \text{ with } y^{q+1} = x^q + x.
$$

It has genus $g(\cH) = q(q-1)/2$, which is largest possible for a maximal function field over $\Fqq$, and any other maximal function field with the same genus is isomorphic to $\cH$ (see \cite{ihara_some_remarks_1982} and \cite{ruck_characterization_1994}). Moreover, the automorphism group of $\cH$ is exceptionally large; it is isomorphic to $\mathrm{PGU}(3,q)$ which has order $q^3(q^2-1)(q^3+1)$.


Any subfield of a maximal function field is again maximal (see \cite{Serre}), and the subfields of $\cH$ corresponding to subgroups of $\mathrm{PGU}(3,q)$ have turned out to be a rich source of examples of maximal function fields (see e.g. \cite{garcia_subfields_2000}). In many cases, these examples come in families of function fields with the same genus, and it is natural to ask for a description of the isomorphism classes within such families. 

It is, in general, a difficult task to determine whether two function fields of the same genus are isomorphic or not. Knowing also the automorphism group is sometimes sufficient, but there are examples of non-isomorphic maximal function fields of the same genus with isomorphic automorphism groups. 

One such example is given by the family of function fields corresponding to the curves studied in \cite{giulietti_m=2_curves_2006}. Let $q$ be a power of an odd prime such that $d = (q+1)/2 > 3$ is prime. Then these function fields are of the form
$$
    \cF_i := \Fqq(x,y) \ \text{ with } y^{q+1} = x^{2i}(x^2 + 1),
$$
for $1 \leq i \leq d-2$. They are subfields of the Hermitian (see \cite[Example 6.4]{garcia_subfields_2000}), and the isomorphism classes and automorphism groups were determined in \cite{giulietti_m=2_curves_2006}. 

Another example of non-isomorphic maximal function fields of the same genus was given in \cite{beelen_families_2024}. For $q$ a power of an odd prime and $d = (q+1)/2$, not necessarily prime, these function fields are of the form
$$
    F_j := \Fqq(x,y) \ \text{ with } y^d = x^j(x^2 + 1),
$$
for $j \in \mZ$ with $\gcd(j(j+2),d) =1$. They are subfields of the Hermitian (see \cite[Example 6.4]{garcia_subfields_2000}, and the isomorphism classes as well as the automorphism groups were described in \cite{beelen_families_2024}, except for the automorphism group of $F_{(d-2)/2}$ which is still unknown. 


In this paper, we will extend most of the results of \cite{giulietti_m=2_curves_2006} to also include the case where $d = (q+1)/2$ is not a prime. In particular, we will determine the isomorphism classes and the full automorphism group for the function fields $\{\cF_i\}_i$, and count the number of isomorphism classes. The structure of the automorphism group is given in Theorem \ref{thm:aut}, while the results regarding the isomorphism classes are collected in Theorem \ref{thm:main_iso_classes} and Theorem \ref{thm:number_iso_classes}.

The core idea is to consider the degree two subfields of $\cF_i$. It turns out that these subfields are, in many cases, isomorphic to fields of the form $F_{j}$, as defined above. Our results then follow from combining the findings of \cite{beelen_families_2024} with a careful study of the automorphism group of $\cF_i$.   

The paper is organized as follows: In Section \ref{sec:the_function_fields} we recall some initial observations regarding the function fields $\{\cF_i\}_i$. This includes a description of some divisors, automorphisms and in particular a number of explicit isomorphisms among the function fields. The rest of the paper is then concerned with showing that no other isomorphisms exist. In Section \ref{sec:subext} we describe the degree two subfields mentioned above, and in Section \ref{sec:weierstrass} we obtain partial results regarding the the Weierstrass semigroups at some special rational places. The automorphism group of $\cF_i$ is completely determined in Section \ref{sec:aut}, and finally the isomorphism classes are described and counted in Section \ref{sec:iso}.

\section{The function fields $\cF_i$}\label{sec:the_function_fields}

Let $q$ be the power of an odd prime and define $d = (q+1)/2$. We study the family of function fields of the form $\mathcal{F}_i := \F_{q^2}(x,y)$ where 
\begin{equation}\label{eq:Fi}
    y^{q+1} = x^{2i}(x^2 + 1),
\end{equation}
for $i\in\mZ$ with $\gcd(i(i+1),d) = 1$. By making the change of variables $y' := ay$, for some $a\in \Fqq$ satisfying $a^{q+1}=-1$, we see that $\cF_i$ belongs to the class the function fields considered in \cite[Example 6.4, Case 2]{garcia_subfields_2000}. It follows that $\cF_i$ is a subfield of the Hermitian function field, and hence $\Fqq$-maximal with $p$-rank zero (see \cite[Lemma 9.73]{hirschfeld_algebraic_2008}). Moreover, the genus of $\cF_i$ is $q-1$, since we are assuming $\gcd(i(i+1),d)=1$. In \cite{giulietti_m=2_curves_2006} these function fields were studied in the case where $d$ is prime. In this section, we recall some properties of $\cF_i$ that hold for any $d$. 


\subsection{Some divisors and special places}\label{sec:divisors_and_omega}

Let $\alpha \in \Fqq$ be some element satisfying $\alpha^2 = -1$. By considering $\cF_i$ as a Kummer extension of $\Fqq(x)$ (see \cite[Proposition 3.7.3]{Sti}), we determine the following divisors in $\cF_i$:
\begin{align}\label{eq:divisors}
\begin{split}
    (x) &= d(P_0^1 + P_0^2) - d(P_\infty^1 + P_\infty^2) \\
    (y) &= i (P_0^1 + P_0^2) + (P_\alpha + P_{-\alpha}) - (i+1)(P_\infty^1 + P_\infty^2), \text{ and } \\
    (dx) &= (d-1) (P_0^1 + P_0^2) + Q (P_\alpha + P_{-\alpha}) - (d+1)(P_\infty^1 + P_\infty^2),
\end{split}
\end{align}

where $P_0^1$ and $P_0^2$ (respectively $P_\infty^1$ and $P_\infty^2$) are the places lying above the zero (respectively pole) of $x$ in $\Fqq(x)$, and $P_\alpha$ (respectively $P_{-\alpha}$) is the place lying above the zero of $(x-\alpha)$ (respectively $(x+\alpha)$). We denote the set of these six places by $\Omega$. \newline

In Section \ref{sec:weierstrass} we will describe the gapsequences of the places of $\Omega$. The key to obtaining this description is the connection between gaps and regular differentials given by the following result:

\begin{proposition}\cite[Corollary 14.2.5]{villa_salvador_topics_2006}\label{prop:reg_diff_gap}
    Let $F$ be an algebraic function field of genus $g$ over some field $K$. Let $P$ be a place of $F$ and $\omega$ a regular differential on $F$. Then $v_P(\omega) + 1$ is a gap at $P$.
\end{proposition}

In the special case $i = 1$ we will use the information on the semigroups to determine the automorphism group of $\cF_1$.

\subsection{First observations regarding the automorphism groups}


We write $\aut(\cF_i)$ for the $\overline{\F}_{q^2}$-automorphism group of $\overline{\F}_{q^2}\cF_i$. Note that this is the same as the $\F_{q^2}$-automorphism group since $\cF_i$ is $\F_{q^2}$-maximal. We immediately find that $\aut(\cF_i)$ contains a subgroup isomorphic to $\mZ_2 \times \mZ_{q+1}$. Indeed, we have
$$
    H_i := \{ \sigma : (x,y) \mapsto (ax,by) \mid a,b\in \F_{q^2}, a^2 = b^{q+1} = 1\} \subseteq \aut(\cF_i).
$$

Note that $d$ is odd since $\gcd(i(i+1),d) = 1$, so $q+1 \equiv 2 \pmod 4$. This means that the unique Sylow $2$-group of $H_i$ is isomorphic to $\mZ_2 \times \mZ_2$. In particular, $H_i$ contains three involutions that give rise to three subfields, $F$, $F'$, and $F''$, of $\cF_i$ satisfying $[\cF_i : F] =[\cF_i : F'] =[\cF_i : F''] = 2$. We will study these subfields further in Section \ref{sec:subext}. In fact, it turns out that they are isomorphic to fields of the type studied in \cite{beelen_families_2024}, and this will be the key to understanding the isomorphism classes of $\{\cF_i\}_i$.

\subsection{Explicit isomorphisms}\label{sec:explicit_iso}

We will determine the isomorphism classes in $\{\cF_i\}_{i}$ by pointing out a number of explicit isomorphisms and then showing that no more isomorphisms exist. The explicit isomorphisms are similar to those described in \cite[Section 7]{giulietti_m=2_curves_2006}:\newline

If $i \equiv j \pmod d$ then $\cF_i$ is isomorphic to $\cF_j$. Indeed, write $j = md + i$ for some $m\in \mZ$, then $\varphi: \cF_i \to \cF_j$ given by $(x,y) \mapsto (x,y/x^m)$ is an isomorphism. Similarly, if $i \equiv -j - 1 \mod d$ then $\cF_i$ is isomorphic to $\cF_j$. The isomorphism is given by $(x,y) \mapsto (1/x, y/x^m)$ where $m\in \mZ$ is chosen such that  $i = md - j - 1$. This means that we can limit ourselves to studying the function fields corresponding to $i = 1, \dots, \frac{d-1}{2}$ where $\gcd(i(i+1),d)=1$.  \newline

Now choose $a\in \Fqq$ such that $a^{q+1} = -1$. We describe the rest of the explicit isomorphisms at the same time: \newline 

If $i,j \in \mZ$ with $\gcd(i(i+1),d) = \gcd(j(j+1),d) = 1$ satisfy either

\begin{alignat*}{2}
    &(1)& \quad ij         &\equiv 1 \pmod d, \\
    &(2)& \quad ij + i + 1 &\equiv 0 \pmod d, \\
    &(3)& \quad ij + i + j &\equiv 0 \pmod d, \text{ or } \\
    &(4)& \quad ij + j + 1 &\equiv 0 \pmod d, 
\end{alignat*}

then $\cF_i$ and $\cF_j$ are isomorphic and an isomorphism from $\cF_i$ to $\cF_j$ is given by respectively

\begin{alignat*}{3}
    &(1)&  \quad(x,y) \mapsto \left(\frac{a^dy^d}{x^{j}}, \frac{a^{i+1}y^{i}}{x^r}\right), \ & \text{ with } r := (ij - 1)/d, \\
    &(2)& \quad (x,y) \mapsto \left(\frac{x^j}{a^dy^d}, \frac{x^r}{a^iy^{i+1}}\right), \ & \text{ with }  r := (ij + i + 1)/d, \\
    &(3)&  \quad(x,y) \mapsto \left(\frac{x^{j+1}}{a^dy^d}, \frac{x^r}{a^iy^{i+1}}\right), \ & \text{ with }  r := (ij + i + j)/d, \text{ and }\\
    &(4)&  \quad (x,y) \mapsto \left(\frac{a^dy^d}{x^{j+1}}, \frac{a^{i+1}y^{i}}{x^r}\right), \ & \text{ with } r := (ij + j + 1)/d.
\end{alignat*}

In Section \ref{sec:iso} we will show that there are no other isomorphisms. For now, note that $(3)$ gives rise to an isomorphism between $\cF_1$ and $\cF_{(d-1)/2}$, so we can limit our considerations to $i = 1, \dots, (d-3)/2$, satisfying $\gcd(i(i+1),2) =1$. We will continue with this simplification throughout the rest of the paper, except in the case $q=5$ where $(d-1)/2 = 1$. We will treat this case separately in the next section, after making some remarks regarding other special cases.

\subsection{The special cases}\label{sec:special} There are two cases where the isomorphisms described above immediately give rise to extra automorphisms. \newline 

If $i^2 + i + 1 \equiv 0 \pmod d$ then the isomorphism from $(2)$ gives rise to an extra automorphism of the form
$$
    \omega: (x,y) \mapsto \left( \frac{x^i}{a^d y^d}, \frac{x^r}{a^iy^{i+1}} \right),
$$
where $r := (i^2 + i + 1)/d$ and $a$ is as above. It can be checked directly that this automorphism has order three, and that it acts as a 3-cycle on the subfields $F$, $F'$, and $F''$. 

Similarly, if $i = 1$ then the isomorphism from $(1)$ gives rise to an extra automorphism 
$$
    \omega_1: (x,y) \mapsto \left( \frac{a^dy^d}{x}, a^2y\right).
$$

By pre-composing with the automorphism $(x,y) \mapsto (\pm x, 1/a^2 y)$ from $H_1$, we obtain two extra involutions in $\aut(\cF_1)$, namely
$$
    \pi : (x,y) \mapsto \left( \frac{a^dy^d}{x},y\right),
$$
and
$$
    \pi' : (x,y) \mapsto \left( -\frac{a^dy^d}{x},y\right).
$$

The case $q=5$ is extra special; we have $d = 3$, so for $i=1$ we get additional automorphisms from both $(2)$ and $(1)$. The genus is $q-1 = 4$, which is equal to second largest possible genus for a maximal curve over $\mathbb{F}_{5^2}$, so $\cF_1$ is isomorphic to the function field $\mathbb{F}_{5^2}(s,t)$ defined by $t^3 = s^5 + s$ (see \cite[Theorem 3.1]{fuhrmann_maximal_1997}). The automorphism group of this function field is known to be a group of order $360 = 60(q+1)$, and it is isomorphic to the semidirect product of a cyclic group of order $3$ and $\mathrm{PGL}(2,5)$ (see \cite[Theorem 12.11]{hirschfeld_algebraic_2008}). The number of isomorphism classes in $\{\cF_i\}_i$ is just one for $q=5$. Since this case is now completely settled, we will often assume $q > 5$ in the following to simplify matters.


\section{Three subfields of $\cF_i$ of degree two}\label{sec:subext}

Assume for the rest of this section that $q > 5$. For a fixed index $i$, satisfying $1\leq i \leq \frac{d-3}{2}$ and $\gcd(i(i+1),d)=1$, we describe the three subfields associated to the involutions of $H_i$. We claim that each of them is isomorphic to a function field of the form $F_j := \F_{q^2}(z,t)$ with
$$
    z^d = t^j(t^2+1),
$$
where $1 \leq j \leq \frac{d-3}{2}$ or $j = d-1$ and $\gcd(j(j+2),d)=1$. These are function fields of the type studied in \cite{beelen_families_2024}. \newline

First, we find a degree two subfield fixed by the involution $\sigma_0:(x,y) \mapsto (x,-y)$. Let $t_0 := y^2$ and note that 
$$
    t_0^d = x^{2i}(x^2+1).
$$
This shows that the subfield $\F_{q^2}(x,t_0) \subseteq \cF_i$ is isomorphic to $F_{2i}$. If $1\leq 2i \leq \frac{d-3}{2}$ we are done since the $\gcd$-condition follows from the $\gcd$-assumption on $i$.

Otherwise, we use the isomorphism from \cite[Lemma 3.2]{beelen_families_2024}: Define $\tilde{x} := 1/x$ and $\tilde{t}_0 := t_0/x$ and note that
$$
    \tilde{t}_0^d = \tilde{x}^{d-2i-2}(\tilde{x}^2+1).
$$
This shows that $\F_{q^2}(x,t_0) = \F_{q^2}(\tilde{x},\tilde{t}_0) \subseteq \cF_i$ is isomorphic to $F_{d-2i-2}$. Since $\frac{d-1}{2} \leq 2i \leq d-3$ (using that $d$ is odd), we have 
$$
    d-2-(d-3) \leq d-2i-2 \leq d-2-\frac{d-1}{2},
$$
i.e.
$$
    1  \leq d-2i-2 \leq \frac{d-3}{2}.
$$

Moreover, 
$$
    \gcd\left((d-2i-2)(d-2i),d\right) = \gcd\left(2i(2i+2),d\right) = \gcd\left(i(i+1),d\right) = 1,
$$
since $d$ is odd. This finishes the proof of the claim for $\sigma_0$. \newline 

For the two other involutions of $H_i$ we need to consider several different cases. Since $\gcd(i(i+1),d)=1$, there is a unique $j \in \{1, \dots, d-1\}$ such that $j$ is an inverse of $i$ modulo $d$. The first two cases depend on whether $j$ is in $\{1, \dots, \frac{d-1}{2}\}$ or in $\{\frac{d+1}{2}, \dots, d-1\}$. Case 3 and 4 depend instead on the inverse of $i+1$ modulo $d$. In each case, the last part of the argument above is needed, but we will not repeat it. \newline

\textbf{Case 1:} Suppose there exists $j\in \mZ$ such that $1\leq j \leq \frac{d-1}{2}$ and  $ij \equiv 1 \pmod d$. If $j = \frac{d-1}{2}$, then $i \equiv 2 \pmod d$, but this is in contradiction with our assumption on $i$, so we may assume $1 \leq j \leq \frac{d-3}{2}$. We now use the isomorphism $(1)$ given in Section \ref{sec:explicit_iso}.

Define $r := \frac{ij-1}{d}$ and pick $a \in \F_{q^2}$ such that $a^{q+1} = -1$. Further, define $x_1 := \frac{a^d y^d}{x^i}$ and $y_1 := \frac{a^{j+1} y^j}{x^r}$. Then, one can check directly that
$$
    y_1^{q+1} = x_1^{2j}(x_1^2 + 1).
$$

Proceeding like above, we define $t_1 := y_1^2$ and obtain a subfield isomorphic to $F_{2j}$. Note that the $\gcd$-condition is satisfied for $2j$ and $2j+2$: \newline

It follows from $ij \equiv 1 \pmod d$ that $\gcd(2j,d)=1$. Since $(j+1)(i+1) \equiv (i + 1) + (j + 1) \pmod d$ and $\gcd((i+1),d)=1$ we also get $\gcd(2j+2,d)=\gcd(j+1,d)=1$. \newline

This means we can copy the argument above and finish the proof of the claim in this case. From the explicit description we see that this subfield is fixed by $\sigma_1:(x,y) \mapsto (-x,y)$ if $i$ is even and $\sigma_2:(x,y) \mapsto (-x,-y)$ if $i$ is odd. \newline

\textbf{Case 2:} Suppose there exists $j_0 \in \mZ$ such that $\frac{d+1}{2} \leq j_0 \leq d-1$ and $ij_0 \equiv 1 \pmod d$. Note that $j_0 = d-1$ would imply $i\equiv -1 \pmod d$ which is impossible since we assume $1\leq i \leq \frac{d-3}{2}$. Using this, we get that $j := d-(j_0+1)$ satisfies
$$
    1\leq j \leq \frac{d-3}{2},
$$
and
$$
    ij + i + 1 \equiv -ij_0 - i + i + 1 \equiv 0 \mod d.
$$

We now use the isomorphism $(2)$ given in Section \ref{sec:explicit_iso}. Define $r := (ij + i + 1)/d$, $a$ like above, $x_2 := \frac{x^i}{a^d y^d}$, and $y_2 := \frac{x^r}{a^j y^{j+1}}$. Then, we have

$$
    y_2^{q+1} = x_2^{2j}(x_2^2 + 1).
$$

Proceeding as before we define $t_2 := y_2^2$ and obtain a subfield isomorphic to $F_{2j}$.  The $\gcd$-condition is satisfied since
$$
    \gcd(2j(2j+2),d) = \gcd(j(j+1),d) = \gcd(j_0(j_0+1),d) = 1,
$$

and we finish with the same argument as previously. Note that this subfield is also fixed by $\sigma_1:(x,y) \mapsto (-x,y)$ if $i$ is even and $\sigma_2:(x,y) \mapsto (-x,-y)$ if $i$ is odd. \newline

\textbf{Case 3:} Suppose there exists $j_0 \in \mZ$ such that $1 \leq j_0 \leq \frac{d-1}{2}$ and $(i+1)j_0 \equiv 1 \pmod d$. Note that $j_0 = 1$ would imply $i \equiv 0 \pmod d$ which is impossible. Using this, we get that $j := j_0-1 $ satisfies
$$
    1\leq j \leq \frac{d-3}{2},
$$
and
$$
    ij + i + j \equiv ij_0 - i + i + j_0 - 1 \equiv 0 \mod d.
$$

We now use the isomorphism $(3)$ given in Section \ref{sec:explicit_iso}. Define $r := (ij + i + j)/d$, $a$ like above, $x_3 := \frac{x^{i+1}}{a^d y^d}$, and $y_3 := \frac{x^r}{a^j y^{j+1}}$. Then, we have

$$
    y_3^{q+1} = x_3^{2j}(x_3^2 + 1).
$$

Proceeding like above we define $t_3 := y_3^2$ and obtain a subfield isomorphic to $F_{2j}$.  The $\gcd$-condition is satisfied since
$$
    \gcd(2j(2j+2),d) = \gcd(j(j+1),d) = \gcd((j_0-1)j_0,d) = \gcd(ij_0^2,d) = 1,
$$

and we are again in a situation where we can easily finish the argument. This subfield is fixed by $\sigma_1:(x,y) \mapsto (-x,y)$ if $i$ is odd and $\sigma_2:(x,y) \mapsto (-x,-y)$ if $i$ is even. \newline

\textbf{Case 4:} Suppose there exists $j_0 \in \mZ$ such that $\frac{d+1}{2} \leq j_0 \leq d-1$ and $(i+1)j_0 \equiv 1 \pmod d$. Now, $j := -j_0+d $ satisfies
$$
    1\leq j \leq \frac{d-1}{2},
$$
and
$$
    ij + j + 1 \equiv -ij_0 - j_0 + 1 \equiv 0 \mod d.
$$

We now use the isomorphism $(4)$ given in Section \ref{sec:explicit_iso}. Define $r := (ij  + j+1)/d$, $a$ like above, $x_4 := \frac{a^d y^d}{x^{i+1}}$, and $y_4 := \frac{a^{j+1} y^j}{x^r}$. Then, we have

$$
    y_4^{q+1} = x_4^{2j}(x_4^2 + 1).
$$

Proceeding like before, we define $t_4 := y_4^2$ and obtain a subfield isomorphic to $F_{2j}$.  The $\gcd$-condition is satisfied since
$$
    \gcd(2j(2j+2),d) = \gcd(j(j+1),d) = \gcd(j_0(1-j_0),d) = \gcd(ij_0^2,d) = 1.
$$

If $\1 \leq 2j \leq \frac{d-3}{2}$ or $2j = d-1$ we are done. Otherwise we copy the argument from previously. Note that this subfield is also fixed by $\sigma_1:(x,y) \mapsto (-x,y)$ if $i$ is odd and $\sigma_2:(x,y) \mapsto (-x,-y)$ if $i$ is even. \newline

By combining all of the above we have proven our claim; each of the three subfields corresponding to the involutions of $H_i$ are isomorphic to a function field of the form $F_j$ where $1 \leq j \leq \frac{d-3}{2}$ or $j = d-1$ and, in both cases, $\gcd(j(j+2),d)=1$. \\

The isomorphism classes in the family $\{F_i\}_i$ were described in \cite{beelen_families_2024}, and we use these results to obtain two useful lemmas: 


\begin{lemma} \label{lemma:iso_subfields_onlyif}
    Assume $i_1$ and $i_2$ satisfy $1\leq i_1,i_2 \leq \frac{d-3}{2}$ and $\gcd(i_1(i_1+1),d)=\gcd(i_2(i_2+1),d)=1$. Let $F'$ be a subfield of $\cF_{i_1}$ associated to an involution of $H_{i_1}$ and let $F''$ be a subfield of $\cF_{i_2}$ associated to an involution of $H_{i_2}$. If $F'$ is isomorphic to $F''$ then either
    \begin{align*}
        i_1i_2 \equiv 0 &\pmod d,\\
        i_1i_2 + i_1 + i_2 \equiv 0 &\pmod d,\\
        i_1i_2 + i_1 + 1 \equiv 0 &\pmod d,\\
        i_1i_2 + i_2 + 1 \equiv 0 &\pmod d,
    \end{align*}
    or we have $i_1 = i_2$.
\end{lemma}

\begin{proof}
    

    For each of $F'$ and $F''$ we can go through the cases mentioned in the above discussion, in combination with Theorem 5.1 and 5.2 from \cite{beelen_families_2024}. This leaves us with only a finite number of cases to check: \newline

    We know that $F'$ is isomorphic to either $F_{2j_1}$ or $F_{d-2j_1-2}$ where either $j_1 = i_1$ or $j_1$ is equal to the $j$ that appeared in one of the four cases discussed above. Similarly, $F''$ is isomorphic to either $F_{2j_2}$ or $F_{d-2j_2-2}$, with $j_2$ equal to $j$ as in one of the four cases or $j_2=i_2$. In any case, the results of \cite{beelen_families_2024} imply that the indices, $2j_1$ or $d-2j_1-2$, and, $2j_2$ or $d-2j_2-2$, must be equal modulo $d$. This amounts to four cases, but in the end it means that either

    \begin{align*}
        j_2 \equiv j_1 &\pmod d, \text{ or }\\
        -j_2-1 \equiv j_1  &\pmod d.\\
    \end{align*}
    
    On the other hand, if we go through the cases above, we see that either

    \begin{align*}
        i_1 \equiv j_1 &\pmod d, &(\text{the } \sigma_0 \text{ case)}\\
        i_1^{-1} \equiv j_1 &\pmod d, &(\text{Case 1})\\
        -i_1^{-1}-1 \equiv j_1 &\pmod d, &(\text{Case 2})\\
        (i_1+1)^{-1} - 1\equiv j_1 &\pmod d,\text{ or } &(\text{Case 3}) \\
        -(i_1+1)^{-1} \equiv j_1 &\pmod d. &(\text{Case 4})\\
    \end{align*}

    We have something similar for $j_2$ (replacing $i_1$ by $i_2$). To finish the proof, one now has to go through all the cases and check that we arrive at one of the equivalences from the statement of the theorem, or $i_1 = i_2$. We give a few examples: \newline

    \begin{itemize}
        \item  If $i_1 \equiv i_2 \pmod d$ then $i_1 = i_2$, since $1 \leq i_1,i_2 \leq \frac{d-1}{2}$. \\
        \item  If $i_1 \equiv i_2^{-1} \pmod d$ then $i_1 i_2 \equiv 1 \pmod d$.\\
        \item If $i_1 \equiv -i_2^{-1} - 1 \pmod d$ then $i_1i_2 + i_2 + 1 \equiv 0 \pmod d$.\\
        \item If $i_1 \equiv (i_2 + 1)^{-1} - 1 \pmod d$ then $i_1i_2 + i_1 + i_2 \equiv 0 \pmod d$.\\
        \item If $i_1 \equiv -(i_2+1)^{-1} \pmod d$ then $i_1i_2 + i_1 + 1 \equiv 0 \pmod d$. \\
        \item If $i_1^{-1} \equiv -i_2^{-1} - 1 \pmod d$ then $i_1i_2 + i_1 + i_2 \equiv 0 \pmod d$.\\
        \item If $i_1^{-1} \equiv (i_2 + 1)^{-1} - 1 \pmod d$ then $i_1i_2 + i_2 + 1 \equiv 0 \pmod d$.\\
        \item If $i_1^{-1} \equiv -(i_2+1)^{-1} \pmod d$ then $i_1 + i_2 + 1 \equiv 0 \pmod d$, but this cannot happen since $1 \leq i_1,i_2 \leq \frac{d-3}{2}$.\\
    \end{itemize}

    The rest of the cases can be treated in a similar way.

\end{proof}


\begin{lemma}\label{lemma:non_iso_conditions}
    Assume $1\leq i \leq \frac{d-3}{2}$ and $\gcd(i(i+1),d)=1$. In $\cF_i$, the three subfields $F$, $F'$, and $F''$, corresponding to the involutions of $H_i$, are pairwise non-isomorphic unless either
    \begin{enumerate}[label=(\alph*)]
        \item $i = 1$, or 
        \item $i^2  + i + 1 \equiv 0 \pmod d$.
    \end{enumerate}
    In the first case, exactly two of the subfields are isomorphic and in the second case all three are isomorphic. Moreover, $F_{d-1}$ is isomorphic to one of the three fields if and only if (a) holds.
\end{lemma}

\begin{proof}
    This follows from considerations very similar to those in the proof of the previous lemma. We show only a few details regarding the special cases: \newline
    \begin{itemize}
        \item If $i = 1$ then $\sigma_0$ fixes a field isomorphic to $F_2$, $\sigma_1$ fixes a field isomorphic to $F_{d-1}$ (this is Case 4 with $j_0 = (d+1)/2$), and $\sigma_2$ fixes a field isomorphic to $F_2$ (this is Case 1 with $j=1$). \newline


        \item If $i^2 + i + 1 \equiv 0 \pmod d$ then there are two cases. If $1 \leq 2i \leq \frac{d-3}{2}$ then $\sigma_0$ fixes $F_{2i}$, we get a field isomorphic to $F_{2i}$ from Case 2 (with $j_0 = d - (i+1)$, and we get another field isomorphic to $F_{2i}$ from Case 4 (here $j_0 = d-i$). Similarly, if $\frac{d-1}{2} \leq 2i \leq d-3$ we get that the three fields are all isomorphic to $F_{d-2i-2}$. \newline
    \end{itemize}
    
    The fact that $F_{d-1}$ does not occur except in case $(a)$ can also be checked by going through the cases: We must have $j = \frac{d-1}{2}$, and this means that we are in Case $4$ with $i=1$.
    
\end{proof}

These two lemmas will be important for determining both the isomorphism classes in $\{\cF_i\}_i$, as well as the automorphism group of each $\cF_i$. We will consider the automorphism groups in Section \ref{sec:aut} and then return to the isomorphism classes in Section \ref{sec:iso}, but first we will need some results on the Weierstrass semigroups at the places of $\Omega$. \newline

\section{The semigroups at the places of $\Omega$}\label{sec:weierstrass}

Instead of considering the Weierstrass semigroups directly, we describe the gapnumbers at the places of $\Omega$. For $i=1$ we show that the gapsequences at $Q_\infty^1$ and $Q_\infty^2$, and hence the semigroups, are distinct from those at the the other places of $\Omega$. This will be useful for determining $\aut(F_1)$ later. First consider $\cF_i = \Fqq(x,y)$, for any $i$ satisfying $\gcd(i(i+1),d) = 1$.\newline

For $k,l \in \mZ$ define the differential $\omega_{k,l} := x^{k-1}y^{l-q-1}dx$. From Equation \ref{eq:divisors} we get


\begin{align*}
    (\omega_{k,l}) =  \ &\left( k d + (l-q-1) i - 1 \right) \left(Q_0^1 + Q_0^2\right) + \left(l-1 \right) \left(Q_\alpha + Q_{-\alpha}\right)\\
        &- \left(kd + (l-q-1)(i+1)  + 1 \right) \left(Q_\infty^1 + Q_\infty^2\right).
\end{align*}

This means that $\omega_{k,l}$ is regular if and only if

\begin{align*}
    l &>0, \\
    kd + li &> i(q+1), \ \text{ and }\\
    kd + (i+1)l &< (i+1)(q+1).
\end{align*}

In other words, $\omega_{k,l}$ is regular exactly if $(k,l)$ is an (integral) interior point of the triangle $\Delta$ with vertices $(0,q+1)$, $(2i,0)$ and $(2(i+1),0)$. Using Pick's theorem and $\gcd((i+1)i,d) = 1$, we find the number of interior integral points of this triangle to be $q-1$, i.e., equal to the genus of $\cF_i$ (as predicted also by well-known results on Newton polygons). \newline

By Proposition \ref{prop:reg_diff_gap}, the regular differentials described above give rise to gap numbers for the places of $\Omega$. The number of distinct differentials equals the number of gaps, i.e., $g(\cF_i) = q-1$, but in some cases two distinct differentials give rise to the same gap number. We will describe the gapsequences completely by considering linear combinations of the $\omega_{k,l}$'s. \newline

Denote by $G_\infty$, $G_0$ and $G_\alpha$ the gapsequences at $Q_\infty^1$, $Q_0^1$ and $Q_\alpha$ respectively. Note that they also equal the gapsequences at $Q_\infty^2$, $Q_0^2$ and $Q_{-\alpha}$, since these pairs of places form orbits under $H_i$. Moreover, denote by $\Delta_1$ the triangle with vertices $(i+1,d)$, $(2i+1,0)$ and $(2(i+1),0)$, and by $\Delta_2$ the triangle with vertices $(i,d)$, $(2i,0)$ and $(2i+1,0)$ (see Figure \ref{fig:1_delta}). We write $\Delta^\circ$ (respectively $\Delta_1^\circ$, $\Delta_2^\circ$) for the interior points of $\Delta$ (respectively $\Delta_1$, $\Delta_2$).  

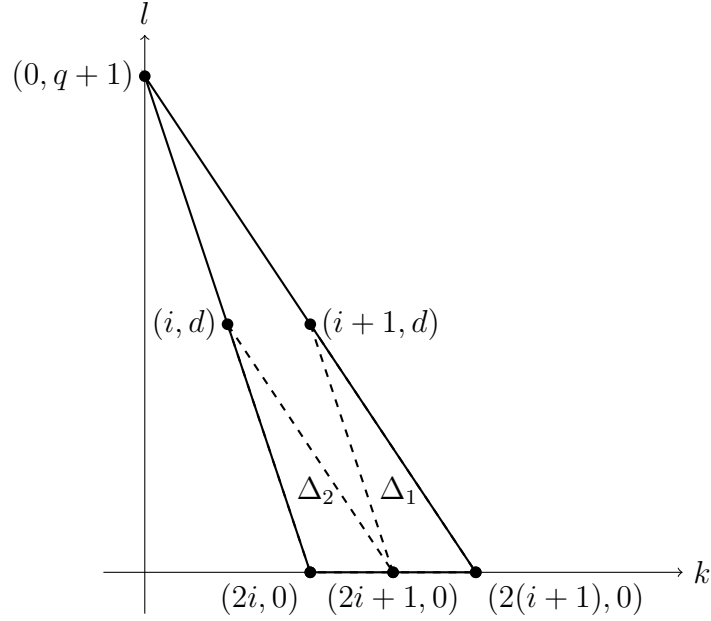
\begin{figure}[h!]

\begin{tikzpicture}[scale = 1.1]

    \draw[->] (-0.5,0) -- (6.5,0) node[right] {$k$}; 
    \draw[->] (0,-0.5) -- (0,6.5) node[above] {$l$}; 

    \def\i{1}
    \def\q{5}
    \def\d{3}

    \coordinate (A) at (0, {\q+1});
    \coordinate (B) at ({2*\i}, 0);
    \coordinate (C) at ({2*(\i+1)}, 0);

    \coordinate (D1) at ({\i+1}, {\d});
    \coordinate (E1) at ({2*\i+1}, 0);
    \coordinate (F1) at ({2*(\i+1)}, 0);

    \coordinate (D2) at ({\i}, {\d});
    \coordinate (E2) at ({2*\i}, 0);
    \coordinate (F2) at ({2*\i+1}, 0);

    \draw[thick] (A) -- (B) -- (C) -- cycle;
    \fill (A) circle (2pt);
    \fill (B) circle (2pt);
    \fill (C) circle (2pt);
    \node[left] at (A) {$(0, q+1)$};
    \node[below left] at (B) {$({2}i, 0)$};
    \node[below right] at (C) {$({2(i+1)}, 0)$};

    \draw[thick, dashed] (D1) -- (E1) -- (F1) -- cycle;
    \fill (D1) circle (2pt);
    \fill (E1) circle (2pt);
    \fill (F1) circle (2pt);
    \node[right] at (D1) {$({i+1}, d)$};
    \node[below] at (E1) {$({2i+1}, 0)$};
    
    \node at ($(D1)!0.45!(E1)!0.40!(F1)$) {$\Delta_1$};

    \draw[thick, dashed] (D2) -- (E2) -- (F2) -- cycle;
    \fill (D2) circle (2pt);
    \fill (E2) circle (2pt);
    \fill (F2) circle (2pt);
    \node[left] at (D2) {$({i}, d)$};
    
    \node at ($(D2)!0.45!(E2)!0.40!(F2)$) {$\Delta_2$};

\end{tikzpicture}

\captionsetup{justification=centering}
\caption{The triangles $\Delta$, $\Delta_1$ and $\Delta_2$.}
\label{fig:1_delta}
\end{figure}

\begin{proposition}\label{prop:semigroups}
    With notation as above, we have
    
    \begin{align*}
        G_\infty = \ &\{-kd - (l-q-1)(i+1) \ \mid \ (k,l) \in \Delta^\circ, l < d \} \\ 
        &\cup \ \{-kd-(l-q-1)(i+1) + q+1 \ \mid \ (k,l) \in \Delta_1^\circ \}, \\
        \\
        G_0 = \ &\{kd + (l-q-1)i \ \mid \ (k,l) \in \Delta^\circ, l < d \} \\
        &\cup \ \{kd + (l-q-1)i + q+1 \ \mid \ (k,l) \in \Delta_2^\circ \}, \text{ and } \\
        \\
        G_\alpha = \ &\{ l \ \mid \ (k,l) \in \Delta^\circ \setminus \Delta_1^\circ \} \ \cup \ \{l + q+1 \mid (k,l) \in \Delta_1^\circ \}.
        \\
    \end{align*}
    
\end{proposition}

\begin{proof}
    We will show details only for the description of $G_\infty$. The results regarding $G_0$ and $G_\alpha$ are obtained in a similar way. \\

    Let $G_1$ be the first set in the union above and $G_2$ the second set. The claim is then that $G_\infty = G_1 \cup G_2$. It follows from Proposition \ref{prop:reg_diff_gap} and the discussion above that the elements of $G_1$ are gap numbers. To see that distinct pairs $(k,l), (k',l') \in \Delta^\circ$, with $l,l'<d$, give rise to distinct gap numbers assume that
    $$
        -kd - (l-q-1)(i+1) = -k'd - (l'-q-1)(i+1).
    $$
    
    Then $kd + l(i+1) = k'd + l'(i+1)$, and working modulo $d$ yields $l = l'$, since $\gcd(i+1,d)=1$ and $l,l' < d$. This implies also $k = k'$, so in fact $(k,l) = (k',l')$. This shows that
    
    $$
        |G_1| = |\{(k,l) \in \Delta^\circ \ \mid \ l<d\}| = q-1 - \frac{q-1}{4},
    $$
    and all these elements are gap numbers at $Q_\infty^1$. \newline

    Now consider instead $G_2$. For $(k,l) \in \Delta_1^\circ$ a direct check shows that $(k-(i+1), l+d)\in \Delta^\circ$. This means that both $\omega_{k,l}$ and $\omega_{k-(i+1), l+d}$ are regular differentials, and so is $\omega := \omega_{k,l}-\omega_{k -(i+1), l + d}$. We determine $v_{Q_\infty^1}(\omega)$ by rewriting
    \begin{align*}
        \omega &= \left(x^{k-1}y^{l-q-1} - x^{k-(i+1)-1}y^{l+d-1}\right) dx \\
            &= \left(1-x^{-(i+1)}y^d\right) x^{k-1}y^{l-1} dx \\
            &= x^{-(i+1)}\left(y^d - x^{i+1}\right) \omega_{k,l} \\
            &= \frac{x^{i-1}}{y^d + x^{i+1}} \omega_{k,l},
    \end{align*}
    
    where the last equality follows from the defining equation of $\cF_i$. This means that 
    \begin{align*}
                v_{Q_\infty^1}(\omega) &= v_{Q_\infty^1}(\omega_{k,l}) + v_{Q_\infty^1}\left(\frac{x^{i-1}}{y^d + x^{i+1}}\right) \\
                &= v_{Q_\infty^1}(\omega_{k,l}) + d(i-1) - d(i+1) \\
                &= -kd-(l-q-1)(i+1)-1 + q+1,
    \end{align*}
    so Proposition \ref{prop:reg_diff_gap} shows that the elements of $G_2$ are in fact gap numbers. A similar argument as for $G_1$ shows that distinct integral points in $\Omega_1^\circ$ give rise to distinct gap numbers, so we have 
    $$
        |G_2| = |\{(k,l) \in \Delta_1^\circ \}| = \frac{q-1}{4}.
    $$

    The total number of gaps is known to be $g(\cF_i) = |G_1| + |G_2|$, so we are done if we can show $G_1 \cap G_2 = \emptyset$. To see that this is true, assume that
    $$
         -kd - (l-q-1)(i+1) = -k'd - (l'-q-1)(i+1) + q+1,
    $$
    for some $(k,l) \in \Delta^\circ$, with $l<d$, and $(k',l') \in \Delta_1^\circ$. Then working modulo $d$ yields $l = l'$ and it follows that $d(k'-k) = q+1$, i.e., $k'-k = 2$. The width of $\Delta^\circ$ is strictly smaller than 2, so this is a contradiction. We conclude that $G_\infty = G_1 \cup G_2$ as desired. \newline

    The results on $G_0$ and $G_\alpha$ are obtained analogously, using differentials of the form $\omega_{k,l} - \omega_{k-i,l+d}$ and $\omega_{k,l}-\alpha \omega_{k-1,l}$ respectively (where as usual $\alpha$ is an element of $\Fqq$ satisfying $\alpha^2 = -1$). 
    
\end{proof}

Even with this rather explicit description it seems difficult to distinguish the gapsequences, or semigroups, at the places of $\Omega$ in general. However, in the special case $i=1$ we are able to do so:

\begin{corollary}\label{cor:semigrous_i=1}
    For $i=1$ and $q > 5$, the gapsequence $G_\infty$ is different from both $G_0$ and $G_\alpha$.
\end{corollary}

\begin{proof}
    We show that $d+2$ is in $G_0$ and $G_\alpha$ but not in $G_\infty$. To see that $d+2 \in G_0$ we check that $(3,2) \in \Delta^0$. Indeed, we have $2 > 0$, $3\cdot d + 2 > q+1$ and $3d + 4 < 2(q+1)$ since $q>5$. Note that also $2 < d$, so it follows from Proposition \ref{prop:semigroups} that $G_0$ contains $3d + (2-q-1) = d + 2$. Similarly, it can be checked that $(1,d+2) \in \Delta^\circ \setminus \Delta_1^\circ$ and this implies $d+2 \in G_\alpha$. \newline

    On the other hand, if $d+2 \in G_\infty$ then, since $d+2 < q+1$, there exists $(k,l)\in \Delta^\circ$ with $l<d$, such that 
    $$
        -kd -2(l-q-1) = d + 2.
    $$

    Working modulo $d$ implies $l = d-1$ and inserting this back into the equation yields $k=1$ as the only option. This is a contradiction since $kd + l = 2d-1 = q$, which shows that $(k,l)=(d-1,1)$ is not an interior point of $\Delta$. The desired result follows.
\end{proof}

In particular, the $\aut(\cF_1)$-orbit containing $Q_\infty^1$ and $Q_\infty^2$ does not contain any other places from $\Omega$. We will use this observation to determine $\aut(\cF_1)$ in the end of the following section. \newline

\section{The automorphism group of $\cF_i$}\label{sec:aut}
We determine the the structure of the automorphism group of $\cF_i$. For convenience, we still assume $1 \leq i \leq \frac{d-3}{2}$, as well as $\gcd(i(i+1),d)=1$ and $q > 5$. As mentioned in the introduction, we already know a subgroup $H_i\subseteq \aut(\cF_i)$,  which is isomorphic to $\mZ_2 \times \mZ_{q+1}$. This means that $H_i$ has a unique Sylow $2$-group, $S$, which is isomorphic to $\mZ_2 \times \mZ_2$. For $i \neq 1$, we will show that $S$ is also the unique Sylow $2$-group of $G$, and use this fact to determine the full automorphism group of $\cF_i$. To complete also the case $i =1 $, we will need the results on the Weierstrass semigroups at the places of $\Omega$. In most cases, we will conclude that there are no more automorphisms than those in $H_i$. \newline

\subsection{The case $i \neq 1$}

In the rest of this section we assume $i \in \{ 2, \dots, (d-3)/2\}$ with $\gcd(i(i+1),d) = 1$. Note that this also implies $q>5$. First, we show that any involution of $\aut(\cF_i)$ is conjugate to one of the three involutions of $H_i$. This will be useful both for determining the full automorphism group of $\cF_i$ and for describing the isomorphism classes, since it implies that any degree two subfield of $\cF_i$ is isomorphic to one of the three described in Section \ref{sec:subext}. 


\begin{theorem}\label{thm:2sylow_is_klein}
    For $i = 2, \dots, (d-3)/2$ with $\gcd(i(i+1),d) = 1$, any involution of $\aut(\cF_i)$ is conjugate to one of the three involutions of $H_i$. 
\end{theorem}

\begin{proof}
    Assume $i \neq 1$. Denote by $S$ the Sylow $2$-group of $H_i$ and by $S_2$ be the Sylow $2$-group of $\aut(\cF_i)$ that contains $S$. Recall that $S$ is isomorphic to $\mZ_2 \times \mZ_2$.  Since $g(\cF_i) = q-1$ is even we can apply \cite[Lemma 6.2]{giulietti_algebraic_many_aut_2019} to obtain a cyclic subgroup of $S_2$ of index 2. \newline

    

    \textbf{Claim 1:} There exists $\varphi \in S$ such that $\varphi$ is central in $S_2$. \newline

    In fact, since $S_2$ is a $2$-group its center is non-trivial and hence contains an element of order $2$, say $\alpha$. Now, if $\alpha \not\in S$ then $\langle \alpha, S\rangle$ is isomorphic to $\mZ_2\times \mZ_2\times \mZ_2$, but this is in contradiction with \cite[Lemma 6.1]{giulietti_algebraic_many_aut_2019} since this $2$-group does not contain a cyclic group of index two. \newline

    \textbf{Claim 2:} $S_2/\langle \varphi \rangle$ has order two. \newline

    Let $F$ denote the fixed field of $\langle \varphi \rangle$. It is a consequence of Galois theory (see \cite[Theorem 11.36]{hirschfeld_algebraic_2008}) that $S_2/\langle \varphi \rangle$ is isomorphic to a subgroup of $\aut(F)$. Now, the automorphism group of $F$ is well understood: From the discussion in Section \ref{sec:subext} we know that $F$ is isomorphic to $F_j$ for some $j \in \mZ$ with $1 \leq j \leq \frac{d-3}{2}$ or $j=d-1$, and $\gcd(j(j+2),d) = 1$. In fact, by Lemma \ref{lemma:non_iso_conditions}, our assumption on $i$ ensures $j\neq d-1$. It follows then, from \cite[Theorem 4.8]{beelen_families_2024} that $\aut(F_j)$ is either cyclic of order $q+1$ or the semidirect product of a cyclic group of order $q+1$ and another cyclic group of order $3$. In any case, since $q \equiv 1 \pmod 4$, this implies the claim. \newline

    It follows from the above that $S_2$ is a group of order four containing (an isomorphic copy of) $\mZ_2\times \mZ_2$, that is $S_2 = S \simeq \mZ_2 \times \mZ_2$. Any other involution $\psi \in \aut(\cF_i)$ is contained in a Sylow 2-group and hence conjugate to an element of $S_2$. This finishes the proof. 
\end{proof}


As an easy consequence we obtain the following:

\begin{corollary}\label{cor:iso_subext}
        For $i = 2, \dots, (d-3)/2$ with $\gcd(i(i+1),d) = 1$, any degree two subfield of $\cF_i$ is isomorphic to one of the three fixed fields of the involutions of $H_i$.
\end{corollary}



We will now distinguish between two different cases. The first case is that in which the three degree two subfields described in Section \ref{sec:subext} are pairwise non-isomorphic. Then, for each Sylow 2-group there are exactly three, pairwise non-isomorphic, degree two subfields arising as fixed fields of the involutions of that group. We will often make use of this, as well as the fact that these three subfields are isomorphic to $F$, $F'$, and $F''$ respectively. In the second case, in which $i^2 + i + 1 \equiv 0 \pmod d$, all three degree two subfields are isomorphic, and we have an extra automorphism $\gamma$ of order three as defined in Section \ref{sec:special}. By Lemma \ref{lemma:non_iso_conditions} this covers everything except $i=1$, which we will deal with separately.

For $i^2 + i + 1 \equiv 0 \pmod d$, we will need the fact that $\omega$ normalizes $H_i$, i.e., that $\langle \omega, H_i\rangle = H_i \rtimes \langle \omega \rangle$. To see this, denote by $F$ a subfield of $\cF_i$ corresponding to an involution of $H_i$. We know from \cite[Theorem 4.8]{beelen_families_2024} that $|\aut(F)| = q+1$, since the characteristic three case does not occur when $i^2 + i + 1 \equiv 0 \pmod d$ (see the comment after Lemma \ref{lemma:number_i^2+i+1_pi(d)}). The degrees match, so the fixed field of $\aut(F)$ is equal to the fixed field of $H_i$ in $\cF_i$. For $h \in H_i$ we have
$$
    \omega^{-1} h \omega \vert_F \in \aut(F).
$$
so $\omega^{-1}h\omega$ fixes the fixed field of $\aut(F)$, which is equal to the fixed field of $H_i$. This means that $\omega^{-1}h\omega \in H_i$, and we conclude that $\langle \omega, H_i \rangle = \langle\omega\rangle \rtimes H_i$ as desired. In particular, $\langle \omega, H_i \rangle$ is a subgroup of $G$ of order $3(q+1)$, and it contains no more involutions than those coming from $H_i$.

Now, we give some further results regarding the involutions and Sylow 2-subgroups of $G$. We know that the involutions of $S$, and hence all the involutions of $G$, fix exactly two places. It turns out that knowing these places is enough to know the involution: 

\begin{lemma}\label{lemma:inv_by_fixed_places}
    For $i = 2, \dots, (d-3)/2$ with $\gcd(i(i+1),d) = 1$, any involution of $G$ is completely determined by the two places it fixes. 
\end{lemma}

\begin{proof}
    Suppose that $\sigma_1,\sigma_2\in G$ are involutions fixing the same places $P$ and $P'$. We claim that $\sigma_1 = \sigma_2$. To show this, first note that both $\sigma_1$ and $\sigma_2$ are in the stabilizer, $G_P$, of $P$. From \cite[Theorem 11.49]{hirschfeld_algebraic_2008} we know that $G_P = S_p \rtimes  C$ where $S_p$ is a $p$-Sylow subgroup of $G_P$ and $C$ is a cyclic subgroup of $G_P$. The characteristic, $p$, is odd by assumption, so $S_p$ has no involutions. Moreover, a cyclic subgroup has at most one involution, so the image of $\sigma_1$ and $\sigma_2$ in $G_P/S_p \simeq C$ must be equal. This means that
    $$
        \sigma_1 \circ \sigma_2 = \sigma_1 \circ \sigma_2^{-1} \in S_p,
    $$
    i.e., $\varphi := \sigma_1 \circ \sigma_2 \in S_p\subseteq G$ is either the identity or has order $p$. 
    
    Recall that the $p$-rank of $\cF_i$ is zero, since $\cF_i$ is $\Fqq$-maximal, so any element of order $p$ has exactly one fixed place (see \cite[Lemma 11.129]{hirschfeld_algebraic_2008}). We know that $\varphi$ fixes both $P$ and $P'$, so it cannot be an element of order $p$. Then, $\varphi$ must be the identity, and we conclude that $\sigma_1 = \sigma_2$, as wished. 
    
\end{proof}

Another important observation is the following:

\begin{lemma}\label{lemma:2syl_trivial_intersection}
    For $i = 2, \dots, (d-3)/2$ with $\gcd(i(i+1),d) = 1$, the intersection of two distinct Sylow $2$-subgroups of $G$ is trivial.
\end{lemma}

\begin{proof}
    Suppose there exists two different Sylow $2$-subgroups with non-trivial intersection. By conjugating with a suitable automorphism we get that $S \subseteq H_i$ has non-trivial intersection with some other Sylow $2$-subgroup $S'$. Pick $\gamma \in G$ such that
    $$
        S' = \gamma^{-1} S \gamma,
    $$

    and consider some $\sigma \in S \cap S'$ different from the identity. Then, find $\sigma_1 \in S$ such that
    $$
        \sigma = \gamma^{-1} \sigma_1 \gamma,
    $$
    and note that the fixed field of $\sigma_1$ must be a degree two subfield of $\cF_i$. Denote this subfield by $F$, and let $F'$ and $F''$ be the two other degree two subfields fixed by elements of $S$. The fixed field of $\sigma$ must also be among these three, since $\sigma \in S$. 
    
    Now, consider the degree two subfield $\gamma^{-1}(F)$. It is easy to check that $\sigma = \gamma^{-1} \sigma_1 \gamma$ fixes all elements of $\gamma^{-1}(F)$. Moreover, the degrees fit so this must be the fixed field of $\sigma$, and hence equal to either $F$, $F'$ or $F''$.

    If the three degree two subfields are pairwise non-isomorphic, the only option is 
    $$
        \gamma^{-1}(F) = F.
    $$

    This means that $\gamma$ restricts to an automorphism on $F$, so $\gamma \in H_i$ and hence 
    
    $$
        S' = \gamma^{-1} S_1 \gamma \subseteq H_i.
    $$
    
    We conclude that $S = S'$, which is a contradiction. \newline 

    If instead all three degree two subfields are isomorphic, we have $i^2 + i + 1 \equiv 0 \pmod d$, and there is an automorphism $\omega \in G$, as described previously, which acts as a $3$-cycle on $F$, $F'$ and $F''$. This means that 
    $$
        \omega^{k} \gamma^{-1} \vert_F \in \aut(F)
    $$
    for some $k \in \{0,1,2\}$, and hence $\omega^k \gamma^{-1} \in H_i$, so $\gamma \in \langle \omega, H_i \rangle = H_i \rtimes \langle \omega \rangle$, which implies $S = S'$. We conclude that distinc Sylow 2-subgroups of $G$ have trivial intersection.

\end{proof}

Finite groups of even order satisfying that different Sylow 2-groups intersect trivially were characterized by M. Suzuki in \cite{suzuki_finite_1964}. Using this, as well as the characterization of certain 2-transitive groups by Kantor, O'Nan and Seitz in \cite{kantor_2-transitive_1972}, we are now able to show a key result regarding the structure of $G$:

\begin{theorem}\label{thm:syl2_is_normal}
    For $i = 2, \dots, (d-3)/2$ with $\gcd(i(i+1),d) = 1$, $S$ is the unique Sylow $2$-subgroup in $G$.
\end{theorem}

\begin{proof}

        If the three degree two subfields are pairwise non-isomorphic then the involutions in $S$ must belong to distinct conjugacy classes. By Lemma \ref{lemma:2syl_trivial_intersection} above we can apply \cite[Lemma 6]{suzuki_finite_1964}, which then implies that $S$ is the unique Sylow $2$-subgroup. \newline
        
        Otherwise, all three degree two subfields are isomorphic, so assume from now on that $i^2 + i + 1 \equiv 0 \pmod d$, and that there is more than one Sylow $2$-subgroup of $G$. \newline
        
        From \cite[Lemma 6]{suzuki_finite_1964} we conclude that all involutions of $G$ are conjugate. By applying Suzuki's classification \cite[Theorem 2]{suzuki_finite_1964} and using $S \simeq \mZ_2\times \mZ_2$ we get that $G$ contains a normal subgroup $G_1$ and $G_2$ such that 
            $$
                \{\text{id}\} \subseteq G_2 \subsetneq G_1 \subseteq G,
            $$
        where both $|G/G_1|$ and $|G_2|$ are odd and $G_1/G_2$ is isomorphic to $A_5$ (the alternating group on five elements). From this we deduce some further results regarding the structure of $G$, which will eventually lead to the contradiction we are searching for. \newline
        



        \textbf{Claim 1:} The number of Sylow $2$-subgroups of $G$ is five. \newline
        
        Let $n_2$ be the number of Sylow $2$-subgroups. From the discussion following Theorem 2 in \cite{suzuki_finite_1964} we see that $G_1/G_2 \simeq A_5$ acts 2-transitively on the set of Sylow $2$-groups of $G$. This immediately implies that $n_2 \leq 6$, since the order of $A_5$ has to be divisible by $n_2(n_2-1)$. On the other hand $A_5$ has five different Sylow 2-subgroups, so we obtain 
        $$
            5 \leq n_2 \leq 6
        $$
        by using that $|G/G_1|$ is odd. By Sylow's theorem $n_2$ is odd, so we conclude that $n_2 = 5$. \newline

        \textbf{Claim 2:} The set $\Omega$ is a $G$-orbit. \newline

        Fix some place $P \in \Omega$. We consider the connection between the number of Sylow 2-subgroups and the size of the $G$-orbit of $P$. Let $\sigma \in H$ be some involution fixing $P$ and another place $P'\in \Omega$, and denote by $O_P$ the $G$-orbit of $P$. For any $\gamma \in \aut(\cF_i)$, we have an involution fixing the places $\gamma(P)$ and $\gamma(P')$, namely 
        $$
            \sigma_\gamma := \gamma \circ \sigma \circ \gamma^{-1}.
        $$
        
        If, for $\gamma_1,\gamma_2 \in G$, we have
        $$
            \{ \gamma_1(P), \gamma_1(P')\} \neq \{\gamma_2(P), \gamma_2(P')\},
        $$ 
        then Lemma \ref{lemma:inv_by_fixed_places} implies that $\sigma_{\gamma_1}$ and $\sigma_{\gamma_2}$ are different involutions. The number of involutions of $G$ is $3\cdot n_2 = 15$, so this means that 
        $$
            15 \geq |O_P|/2.
        $$

        Recall that $H_i$ acts with long orbits outside of $\Omega$, so 
        $$
            |O_P| = 6 + 2k (q+1) \leq 30,
        $$
        which is true only if $k=0$ or $q \leq 11$. Now, the only options for $q \leq 11$ are $q = 5$ and $q=9$. In the first case we must have $i = 1$, so this option is not valid, and in the second case the equation $i^2 + i + 1 \equiv 0 \pmod d$ has no solutions, so this case does not occur. We conclude that $k = 0$, so in fact $O_P = \Omega$. \newline

        \textbf{Claim 3:} $G$ acts 2-transitively on $\Omega$. \newline
        

        The number of involutions is $15 = \binom{6}{2}$, they are all in the same conjugacy class and any involution fixes exactly two places in $\Omega$. This means there is a 1-to-1 correspondence between pairs of places of $\Omega$ and involutions of $G$. Now fix some $P \in \Omega$ and choose $P' \in \Omega$ such that $\{P,P'\}$ forms an $H_i$-orbit. Let $\pi \in H_i$ be some automorphism switching $P$ and $P'$, and let $\sigma$ be the involution that fixes $P$ and $P'$. For a place $Q \in \Omega \setminus \{P,P'\}$ denote by $\sigma'$ the involution fixing $P$ and $Q$, and determine $\gamma \in G$ such that
        $$
            \sigma' = \gamma \sigma \gamma^{-1}.
        $$

        Then $\gamma$ maps $\{P, P'\}$ to $\{ P, Q\}$, so either $\gamma$ fixes $P$ and maps $P'$ to $Q$ or $\gamma \circ \pi$ fixes $P$ and maps $P'$ to $Q$. This shows that the stabilizer of $P$ acts transitively on $\Omega \setminus \{P\}$, so we conclude that $G$ acts 2-transitively on $G$. \newline

        Finally, we will use the classification by Kantor, O'Nan and Seitz in \cite{kantor_2-transitive_1972} to obtain a contradiction. Note that the stabilizer of two different places in $\Omega$ is cyclic by \cite[Theorem 11.49]{hirschfeld_algebraic_2008} and \cite[Lemma 11.129]{hirschfeld_algebraic_2008}, since the $p$-rank of $\cF_i$ is zero. This means we can apply the classification result \cite[Theorem 1.1]{kantor_2-transitive_1972}. Since the order of $\Omega$ is not a prime power, $G$ cannot have a regular normal subgroup (see e.g. \cite[Theorem 1.7.5]{biggs_permutation_1979}), so $G$ must be one of the groups
        $$
            \mathrm{PSL}(2,q_0), \ \mathrm{PGL}(2,q_0), \ \mathrm{PSU}(3,q_0), \ \mathrm{PGU}(3,q_0), \ \mathrm{Sz}(q_0), \text{ or } \mathrm{Ree}(q_0),
        $$
        where $q_0$ is a prime power. We know $|G|$ is divisible by four but not eight, and this is enough to exclude $\mathrm{PSU}(3,q_0)$,  $\mathrm{PGU}(3,q_0)$ and $\mathrm{Ree}(q_0)$. Also, the only option for $\mathrm{Sz}(q_0)$ is $q_0 = 2$, but in this case three does not divide the order. The group $\mathrm{PGL}(2,q_0)$ has order divisible by eight except for $q_0 = 2$ and $q_0 = 4$, but $G \simeq \mathrm{PGL}(2,2)$ or $G \simeq \mathrm{PGL}(2,4)$ would imply 
        $$
            6(q+1) \leq |G| \leq 60,
        $$
        which only happens for $q \leq 9$, and we already saw that $q = 5$ and $q = 9$ does not occur.

        A similar argument shows that $G \simeq \mathrm{PSL}(2,q_0)$ cannot happen for $q_0$ even. If $q_0$ is odd, then the number of involutions of $\mathrm{PSL}(2,q_0)$ is known to be $q_0(q_0-1)/2$ (see, e.g., \cite[Section 13, Theorem 1.4 and the beginning of Subsection 13.3]{gorenstein1980finite}), and this is not equal to $15$ for any valid choice of $q_0$.

        There are no more remaining options, so we have arrived at a contradiction. We conclude that $S$ is the unique Sylow $2$-subgroup of $G$ as desired.

\end{proof}

The description of the full automorphism group now follows easily: 

\begin{corollary}
    For $i = 2, \dots, (d-3)/2$ with $\gcd(i(i+1),d) = 1$ we have
    $$
        \aut(\cF_i) = 
        \begin{cases}
            H_i \rtimes \langle \omega \rangle &\text{ if } \ i^2 + i + 1 \equiv 0 \pmod d, \text{ and } \\
            \hfil H_i &\text{ otherwise.}
        \end{cases}
    $$
\end{corollary}

\begin{proof}
    For $\sigma \in G$, it follows from Theorem \ref{thm:syl2_is_normal} that $\sigma(F)\in \{F, F', F''\}$. We consider the two different cases. \newline
    
    Assume first that $i^2 + i + 1 \not\equiv 0 \pmod d$. Then $F$, $F'$ and $F''$ are pairwise non-isomorphic, so the only option is $\sigma(F) = F$. This means that $\sigma\vert_F \in \aut(F)$. From \cite[Theorem 4.8]{beelen_families_2024} we know $|\aut(F)| = q+1$ unless $F \simeq F_1$ and $q$ is a power of three. In this case, replace $F$ by $F'$, and note that $F' \not\simeq F_1$. Since the degrees match, the fixed field of $\aut(F)$, or $\aut(F')$, must be equal to the fixed field of $H_i$ in $\cF_i$. In particular, $\sigma$ fixes the fixed field of $H_i$, and hence  $\sigma \in H_i$. Since $\sigma$ was arbitrary this shows $G = H_i$. \newline

    If instead $i^2 + i + 1 \equiv 0 \pmod d$ then $F$, $F'$ and $F''$ are all isomorphic, and $\aut(\cF_i)$ contains an automorphism, $\omega$, which acts as a $3$-cycle on $\{F,F',F''\}$. In particular, 
    $$
        \omega^k \sigma \vert_F \in \aut(F),
    $$
    for some $k \in \{0,1,2\}$. From \cite[Theorem 4.8]{beelen_families_2024} we know $|\aut(F)| = q+1$, so again the fixed field of $\aut(F)$ is equal to the fixed field of $H_i$. This implies that $\omega^k \sigma \in H_i$, so $\sigma \in \langle \omega, H_i \rangle = H_i \rtimes \langle \omega \rangle$, and this finishes the proof.
    

\end{proof}

\subsection{The case $i=1$}\label{sec:special_i=1} The previously used methods appear to be inadequate in this case. One reason is that the automorphism group now contains more involutions. Another, is that one of the subfields arising from the involutions of $H_1$ is $F_{d-1}$, which is isomorphic to the Roquette curve and hence has a large automorphism group. Instead, we will rely on information regarding the Weierstrass semigroups at the places of $\Omega$, and use a method similar to what was done in \cite{beelen_families_2024}. \newline



We claim that $\aut(\cF_1)$ is generated by $\pi$ and $H_1$, where $\pi$ is the involution defined in Section \ref{sec:special}. In fact, we have the following theorem:

\begin{theorem}
    For $q > 5$ and $i=1$, the automorphism group of $\cF_i$ is the semidirect product of $H_i$ and a group of order two. In particular, we have $|\aut(\cF_i)| = 4(q+1)$.
\end{theorem}

\begin{proof}
    Define $G := \aut(\cF_1)$ and $g := g(\cF_1) = q-1$. Direct calculations show that $\langle H_1, \pi \rangle = H_1 \rtimes \langle \pi \rangle$, so $|G| \geq 4(q+1)$, and the theorem follows if we can show $|G| \leq 4(q+1)$. We check the result directly with a computer for $q < 37$, and for $q \geq 37$ we proceed by considering the orbit of $Q_\infty^1$: \newline
    
    Assume from now on that $q\geq 37$, and denote by $O_\infty$ the $G$-orbit containing both $Q_\infty^1$ and $Q_\infty^2$. By Corollary \ref{cor:semigrous_i=1} it cannot contain any other places from $\Omega$. If the orbit is of length more than two then, since $H_1$ acts with long orbits outside of $\Omega$, the orbit-stabilizer theorem yields
    \begin{align*}
        |G| = |O_\infty| \cdot |\aut(\cF_1)_{Q_\infty^1}| \geq (2 + 2(q+1)) (q+1) = (2g + 6)(g+2) > 84(g-1),
    \end{align*}
    because $q \geq 37$. Hence \cite[Theorem 11.56]{hirschfeld_algebraic_2008} applies, so $|G|$ is divisible by the characteristic $p$, and one of the following cases holds:
    
    \begin{enumerate}
        \item $G$ has exactly one short orbit,
        \item $G$ has exactly three short orbits, of which two have cardinality $|G|/2$, or
        \item $G$ has exactly two short orbits, of which at least one is non-tame, i.e., the order of the stabilizer of a place in the orbit is divisible by $p$.
    \end{enumerate}

    All places of $\Omega$ have a non-trivial stabilizer (they each contain a cyclic subgroup of $H_1$ of order $(q+1)$), so they must be contained in short orbits of $G$. This immediately excludes the first case because of Corollary \ref{cor:semigrous_i=1}. The second case also cannot occur; the stabilizers of each place in $\Omega$ is of order at least $q+1$, so this would again imply that all places of $\Omega$ are in the same orbit. We are left with Case (3): \newline
    
    Assume that $G$ gives rise to exactly two short orbits, $O_1$ and $O_2$, and that at least one of them, say $O_1$, is non-tame. The places of $\Omega$ cannot all be in the same orbit, again by Corollary \ref{cor:semigrous_i=1}, so there exists some $P \in \Omega \cup O_1$. By \cite[Theorem 11.49]{hirschfeld_algebraic_2008} we may write 
    $$
        \aut(\cF_1)_{P} = S_p \rtimes C,
    $$
    where $S_p$ is a Sylow $p$-subgroup of $\aut(\cF_1)_{P}$ and $C$ is cyclic or order not divisible by $p$. Note that the cyclic subgroup of $H_i$ which fixes $P$ is contained in $C$, so the order of $C$ is a multiple of $q+1$. Now, define $E_P$ to be the fixed field of $S_P$ in $\cF_1$, so that $\overline{C} := \aut(\cF_1)/S_p \simeq C$ is a cyclic subgroup of $\aut(E_P)$. We consider three different cases, depending on the genus of $E_P$: \newline

    \textbf{Case 1:} Assume $g(E_P) \geq 2$. Then we can apply \cite[Theorem 11.79]{hirschfeld_algebraic_2008} to obtain
    $$
        q+1 \leq |C| \leq 4g(E_P) + 4.
    $$
    On the other hand, the Riemann-Hurwitz formula applied to the extension $\cF_1/E_P$ yields
    $$
        2g - 2 \geq |S_P| (2g(E_P)-2) + (|S_P|-1).
    $$
    From combining the above we get
    $$
        q+1 \leq |C| \leq \frac{4q - 6}{|S_P|} + 6,
    $$
    which in turn implies $|S_P| < 5$, since $q \geq 37$.  

    Hence, only the case $|S_P| = p = 3$ remains, and in this case we have $|C| < \frac{4q-6}{3} -2 < 2(q+1)$. Since $|C|$ is a multiple of $q+1$, this implies $|C| = q+1$ so that $C\subseteq H_1$. Now, consider a generator $\tau$ of $S_3$. By definition $\tau$ fixes $P$, and since the $p$-rank of $\cF_1$ is zero it fixes no other places by \cite[Lemma 11.129]{hirschfeld_algebraic_2008}. In particular, $\tau$ acts with orbits of length three on the remaining five places of $\Omega$, so there must be a $\tau$-orbit containing both a place from $\Omega$ and a place not in $\Omega$. This is a contradiction since $C$ acts on the $S_P$-orbits, and $C$ acts with orbits of length at most two on places of $\Omega$ and orbits of length $q+1$ everywhere else. \newline

    \textbf{Case 2:} Assume $g(E_P) = 1$. Then \cite[Remark 11.95]{hirschfeld_algebraic_2008} implies that $q < 13$, but we are assuming $q \geq 37$. \newline

    \textbf{Case 3:} Assume $g(E_P) = 0$. Then \cite[Theorem 11.91]{hirschfeld_algebraic_2008} implies that $\overline{C}$ fixes exactly two places of $E_P$ and acts with long orbits everywhere else. This means that the cyclic group $H':= H_1 \cap C$ fixes exactly two $S_P$-orbits. One of them is $\{P\}$ and the other one must contain anything with a nontrivial $H'$-stabilizer. In particular, all the remaining places of $\Omega$ must be in the same $S_P$-orbit, and hence all of $\Omega$ is in the same $G$-orbit, but this is in contradiction with Corollary \ref{cor:semigrous_i=1}. \newline

    We obtain a contradiction in all cases, so we conclude that $O_\infty = \{Q_\infty^1, Q_\infty^2\}$. By the orbit-stabilizer theorem this implies 
    $$
        |G| = 2 |S|,
    $$
    where $S := \aut (\cF_1)_{Q_\infty^1}$. We know that $S$ contains a cyclic subgroup $H' := H_i \cap S$ of order $q+1$, and we will finish the proof by showing $|S| \leq 2|H'| = 2(q+1)$. \newline
    
    First note that the elements of $S$ fix both places in $O_\infty = \{Q_\infty^1, Q_\infty^2\}$. From \cite[Lemma 11.129]{hirschfeld_algebraic_2008} we therefore get that $S$ contains no element of order $p$, and it follows both that $G$ is tame and that $S_P$ is cyclic (by \cite[Theorem 11.49]{hirschfeld_algebraic_2008}). Now, consider a generator $\beta$ of $S$. Since $S$ is cyclic $H'$ is normal in $S$, so $S$ acts on the orbits of $H'$. In particular, $S$ acts on the set of short $H'$-orbits $\left\{ \{Q_0^1,Q_0^2\},\{Q_\alpha, Q_{-\alpha}\}\right\}$. It follows that $\beta^2$ fixes  the divisor of both $x$ and $y$, so we must have 
    $$
        \beta(x) = \lambda x \ \text{ and } \ \beta(y) = \mu y,
    $$
    for some $\lambda, \mu \in \Fqq$. From the defining equation of $\cF_1$ we obtain
    $$
        \mu^{q+1} y^{q+1} = \mu^{q+1} x^2(x^2 + 1) = \lambda^2 x^2(\lambda^2 x^2 + 1), 
    $$
    which is only possible if $\mu^{q+1} = \lambda^2 =  1$. We conclude that $\beta^2 \in H_1$, and since $\beta^2 \in S$ by definition, this shows $\beta^2 \in H'$. Finally, this implies 
    $$
        |G| = 2\cdot|S| \leq 2\cdot (2\cdot|H'|) = 4(q+1),
    $$
    as desired. We conclude that $|G| = 4(q+1)$ which means $G = \langle H_1, \pi\rangle = H_1 \rtimes \langle \pi \rangle$, and this finishes the proof.  
\end{proof}

We sum up the results regarding automorphism groups in the following theorem:

\begin{theorem}\label{thm:aut}
    Let $q$ be the power of an odd prime with $q > 5$, and suppose $1 \leq i \leq (d-3)/2$ with $\gcd(i(i+1),d)=1$. Then, keeping the notation from previously, the automorphism group of $\cF_i$ is given by 
    $$
        \aut(\cF_i) = 
        \begin{cases}
            H_i \rtimes \langle \pi \rangle & \text{ if } \ i=1, \\
            \hfil H_i \rtimes \langle \omega \rangle &\text{ if } \ i^2 + i + 1 \equiv 0 \pmod d, \text{ and } \\
            \hfil H_i &\text{ otherwise.}
        \end{cases}
    $$
    In particular, the order of the automorphism group is $4(q+1)$ if $i=1$, $3(q+1)$ if $i^2 + i + 1 \equiv 0 \pmod d$ and $q+1$ otherwise.
\end{theorem}

\section{Isomorphism classes}\label{sec:iso}

We determine the isomorphism classes among $\{\cF_i\}_i$ and calculate the number of distinct isomorphism classes. Note that the results are in accordance with the findings of \cite{giulietti_m=2_curves_2006} when $d$ is a prime. The main result is the following:

\begin{theorem}\label{thm:main_iso_classes}

    For $1 \leq i_1 < i_2 \leq \frac{d-1}{2}$ with $\gcd(i_1(i_1+1),d)=\gcd(i_2(i_2+1),d) = 1$, the function fields $\cF_{i_1}$ and $\cF_{i_2}$ are isomorphic if and only if
    
    \begin{align*}
        i_1i_2 \equiv 0 &\pmod d,\\
        i_1i_2 + i_1 + i_2 \equiv 0 &\pmod d,\\
        i_1i_2 + i_1 + 1 \equiv 0 &\pmod d, \text{ or }\\
        i_1i_2 + i_2 + 1 \equiv 0 &\pmod d.\\
    \end{align*}
\end{theorem}

\begin{proof}

    For $q=5$ there is nothing to show, so assume from now on that $q>5$. The ``if'' part is covered by the explicit isomorphisms given in Section \ref{sec:explicit_iso}.
    
    
    The ``only if'' part follows from combining Theorem \ref{thm:aut} and Lemma \ref{lemma:iso_subfields_onlyif}. In fact, suppose that $\cF_{i_1}$ and $\cF_{i_2}$ are isomorphic. We consider three different cases: \newline
    
    \textbf{Case 1:} If $i_1 = 1$, then it follows from Theorem \ref{thm:aut} that $i_2 = \frac{d-1}{2}$, and we have $i_1i_2+i_1+i_2 \equiv 0 \pmod d$. \newline
    
    \textbf{Case 2:} If $i_1^2 + i_1 + 1 \equiv 0 \pmod d$, then it follows from Theorem \ref{thm:aut} that also $i_2^2 + i_2 + 1 \equiv 0 \pmod d$, and hence that the only involutions in $\aut(\cF_{i_1})$ and $\aut(\cF_{i_2})$ are those coming from $H_{i_1}$, respectively $H_{i_2}$. Applying Lemma \ref{lemma:iso_subfields_onlyif} now gives the desired result. In fact, it follows from the discussion in the proof of Lemma \ref{lemma:non_iso_conditions} that $i_1 = i_2$. \newline

    \textbf{Case 3:} Otherwise, it follows from Theorem \ref{thm:aut} that $\aut(\cF_{i_1}) = H_{i_1}$, and hence also $\aut(\cF_{i_2}) = H_{i_2}$. Applying Lemma \ref{lemma:iso_subfields_onlyif} now gives the desired result.    

\end{proof}

The number of isomorphism classes in $\{\cF_i\}_i$ hence depends on the number of distinct solutions to $i^2 + i + 1 \equiv 0 \pmod d$. We determine this number in terms of the prime facotization of $d$.

\begin{lemma}\label{lemma:number_i^2+i+1_pi(d)}
    Assume $q>5$. Write $d = p_1^{\alpha_1}\cdots p_n^{\alpha_n}$ for distinct odd primes $p_1, \dots , p_n$ and $\alpha_1, \dots, \alpha_n \in \mZ_{\geq 0}$. Let $m_1$ (respectively $m_2$) be the number of primes among $p_1, \dots, p_n$ congruent to one (respectively two) modulo three. Then, the number of distinct solutions to $i^2 + i + 1 \equiv 0 \pmod d$ in $\{1, \dots, \frac{d-3}{2}\}$ is
    $$
        \pi(d) = 
        \begin{cases}
            0 &\text{if } 9\mid d \text{ or } m_2 \geq 1, \\
            2^{m_1 - 1} &\text{otherwise.}
        \end{cases}
    $$
\end{lemma}

\begin{proof}



    We first count solutions for $i\in \{0, \dots, d-1\}$. By the Chinese Remainder Theorem this can be reduced to counting solutions of $i^2 + i + 1 \equiv 0 \pmod{p^k}$ in $\{0,\dots, p^k-1\}$, for $p$ in $\{p_1, \dots, p_n\}$. 
    
    If $p = 3$ and $k=1$ there is exactly one solution, namely $i=1$. A direct check shows that $i^2 + i + 1 \equiv 0 \pmod 9$ never holds, so if $p = 3$ and $k \geq 2$ there are no solutions. 
    
    Suppose $p>3$, and note that then $i \equiv 1 \pmod p$ is never a solution. Since $(i^3-1) = (i-1)(i^2+i+1)$ this means that the solutions of $i^2 + i + 1 \equiv 0 \pmod{p^k}$ in $\{0,\dots, p^k-1\}$ correspond to elements of order three in $\left(\mZ/p^k\mZ\right)^\times$. This group is cyclic of order $p^{k-1}(p-1)$, so there are no elements of order three if $p \equiv 2 \pmod 3$, and exactly two elements of order three if $p \equiv 1 \pmod 3$. We conclude that the number of solutions to $i^2 + i + 1 \equiv 0 \pmod d$ in $\{0, \dots, d-1\}$ is zero if $9\mid d$ or $m_2 \geq 1$, and $2^{m_1}$ otherwise.

    To finish the proof, note that if $i^2 + i + 1 \equiv 0 \pmod d$ then $d-(i+1)$ is another solution. We assume $q > 5$, so this means that the solutions to $i^2 + i + 1 \equiv 0 \pmod d$ among $\{1, \dots, d-1\}$ come in pairs, with exactly one member of each pair being in $\{1, \dots, \frac{d-3}{2}\}$. The desired result now follows.
\end{proof}

As an easy consequence, we note that if $q$ is a power of $3$ then $d \equiv 2 \pmod 3$, so it is divisible by at least one prime congruent to $2$ modulo $3$, and hence $i^2 + i + 1 \equiv 0 \pmod d$ has no solutions. \newline

The number of isomorphism classes can now be determined:

\begin{theorem}\label{thm:number_iso_classes}
    Let $q > 5$ be the power of a prime with $q \equiv 1 \pmod 4$, $d := (q+1)/2$ odd, and $\{\cF_i\}_i$ as defined in Equation \ref{eq:Fi}. Write $d = p_1^{\alpha_1}\cdots p_n^{\alpha_n}$ for distinct odd primes $p_1, \dots , p_n$ and $\alpha_1, \dots, \alpha_n \in \mZ_{\geq 0}$. The number of isomorphism classes among the function fields $\{\cF_i\}_{i}$ is
    $$
        N(d) = \frac{\varphi_2(d) + 4\pi(d) + 3}{6},
    $$
    where $\pi(d)$ is as defined in Lemma \ref{lemma:number_i^2+i+1_pi(d)} and
    $$
        \varphi_2(d) = p_1^{\alpha_1-1}(p_2-2) \cdots p_n^{\alpha_n - 1}(p_n - 2).
    $$
\end{theorem}

\begin{proof}
    We follow the same strategy as in \cite[Theorem 5.3]{beelen_families_2024}. Using the first explicit isomorphism mentioned in Section \ref{sec:explicit_iso} we reduce the problem to counting isomorphism classes for $i \in \{0,1, \dots, d-1\}$. Among these numbers, there are exactly $\varphi_2(d)$ choices for $i$ such that $\gcd(i(i+1),d)=1$. By the second isomorphism mentioned in Section \ref{sec:explicit_iso}, we can reduce this to $\frac{\varphi_2(d)+1}{2}$ valid choices for $i \in \{0, \dots, \frac{d-1}{2}\}$. Now consider the function fields arising from these choices of $i$ only. Then, as noted in the proof of the above theorem, the set $\{\cF_1,\cF_{(d-1)/2}\}$ is an isomorphism class, and so is $\{\cF_i\}$ for any $i$ satisfying $i^2 + i + 1 \equiv 0 \pmod d$. We claim that all other isomorphism classes will have size three:
    
    In fact, suppose $i\in \{2, \dots, \frac{d-3}{2}\}$ satisfies $\gcd(i(i+1),d) = 1$ and $i^2 + i + 1 \not\equiv 0 \pmod d$. Then Theorem \ref{thm:main_iso_classes} shows $\cF_i$ is isomorphic to $\cF_{i'}$ and $\cF_{i''}$, where $i'$ is equal to $j$ as defined in either Case 1 or Case 2 from Section \ref{sec:subext} (depending on the inverse of $i \pmod d$), and $i''$ is equal to $j$ as defined in either Case 3 or Case 4 (depending on the inverse of $i+1 \pmod d$). The assumptions on $i$ guarantee that $i$, $i'$ and $i''$ are distinct. \newline

    From the above observations we conclude that the number of isomorphism classes is 
    
    $$
        1 + \pi(d) + \frac{1}{3}\left(\frac{\varphi_2(d)+1}{2} - 2 - \pi(d) \right) =  \frac{\varphi_2(d) + 4\pi(d) + 3}{6}, 
    $$
    
    where $\pi(d)$ is as defined in Lemma \ref{lemma:number_i^2+i+1_pi(d)}.  
\end{proof}

Note that $\varphi_2(d) = d-2$ when $d$ is prime, so in this case we recover the formula for $N(d)$ given in \cite[Theorem 1.1]{giulietti_m=2_curves_2006}.

\section*{Acknowledgements}
The author would like to thank Maria Montanucci and Peter Beelen for helpful discussions and insightful suggestions throughout the work that led to this paper.

\bibliographystyle{abbrv}
\bibliography{ref}

\end{document}